\documentclass[12pt,reqno]{amsart}

\usepackage[utf8]{inputenc}
\usepackage{amsfonts,amsmath,amssymb,amsthm}
\usepackage{enumitem}
\usepackage{amsaddr}
\usepackage{mathrsfs}

\newtheorem{Theorem}{Theorem}[section]
\newtheorem{Lemma}[Theorem]{Lemma}
\newtheorem{Proposition}[Theorem]{Proposition}

\newtheorem{Remark}{Remark}[section]
\newtheorem{Definition}{Definition}[section]
		
\numberwithin{equation}{section}

\usepackage[a4paper,margin=3cm]{geometry}

\begin{document}
\sloppy

\title[The Cantor-Bendixson derivative on a Polish space] 
{Some properties related to the Cantor-Bendixson derivative 
on a Polish space}

\author[B. \'Alvarez-Samaniego]{Borys \'Alvarez-Samaniego}
\address{\vspace{-7mm}N\'ucleo de Investigadores Cient\'{\i}ficos\\
	Facultad de Ciencias\\
	Universidad Central del Ecuador (UCE)\\
	Quito, Ecuador}
\email{balvarez@uce.edu.ec, borys\_yamil@yahoo.com}

\author[A. Merino]{Andr\'es Merino}
\address{\vspace{-7mm}Escuela de Ciencias F\'isicas y Matem\'atica\\
    Facultad de Ciencias Exactas y Naturales\\
    Pontificia Universidad Cat\'olica del Ecuador (PUCE)\\
    Quito, Ecuador}
\email{aemerinot@puce.edu.ec}

\date{February 28, 2020}

\begin{abstract}
We show a necessary and sufficient condition for any ordinal number 
to be a Polish space. We also prove that for each countable 
Polish space, there exists a countable ordinal number that is an upper 
bound for the first component of the Cantor-Bendixson characteristic of 
every compact countable subset of the aforementioned space. In addition, 
for any uncountable Polish space, for every countable ordinal number 
and for all nonzero natural number, we show the existence of a compact 
countable subset of this space such that its Cantor-Bendixson 
characteristic equals the previous pair of numbers. Finally, for each 
Polish space, we determine the cardinality of the partition, 
up to homeomorphisms, of the set of all compact countable 
subsets of the aforesaid space. 
\end{abstract}


\maketitle

\section{Introduction} \label{Sec:Intro}
First, we give some notations, definitions and basic facts that will 
be useful throughout this paper. 
\begin{Definition}[Metrizable space]
A topological space $(E,\tau)$ is called \emph{metrizable} if there exists 
a metric on $E$ that generates $\tau$.
\end{Definition}
By $\sim$ we denote homeomorphism between topological spaces. We have 
that the topological space $(E,\tau)$ is metrizable if and only if 
there is a metric space $(\widetilde E,d)$ such that $E \sim \widetilde E$.  
In fact, if ${f}\colon{E}\to{\widetilde E}$ is a homeomorphism from $E$ onto 
$\widetilde{E}$, we can take the metric $d_E$ on $E$ given by 
\begin{equation*}
    \setlength{\arraycolsep}{2pt}
	\begin{array}{r@{}ccl}
	d_E\colon & E\times E & \longrightarrow & \mathbb{R}\\
	& (x,y) & \longmapsto & d_E(x,y)=d(f(x),f(y)).
    \end{array}
\end{equation*}
\begin{Definition}[Completely metrizable space]
We say that a topological space $(E,\tau)$ is \emph{completely metrizable} 
if there exists a metric $d$ on $E$ that generates $\tau$ and $(E,d)$ is  
a complete metric space.
\end{Definition}
Furthermore, we see that the topological space $(E,\tau)$ is completely 
metrizable if and only if there exists a complete metric space 
$(\widetilde E,d)$ such that $E \sim \widetilde E$.
\begin{Definition}[Polish space]
A topological space $(E,\tau)$ is \emph{Polish} if it is separable 
and completely metrizable.
\end{Definition}
We use $\mathcal{P}(D)$ and $|D|$ to symbolize, respectively, the 
power set and the cardinality of the set $D$. If $(Y,\tau)$ is a 
topological space, then $\mathcal{K}_Y$ stands for the set of all 
compact countable subsets of $Y$, where a countable set is either 
a finite set or a countably infinite set, and 
$\mathscr{K}_Y := \mathcal{K}_Y /\!\sim$ represents the set of all the 
equivalence classes, up to homeomorphisms, of elements of $\mathcal{K}_Y$.  
If $(G,d)$ is a metric space, $x \in G$ and $r>0$, we write $B(x,r)$ to 
indicate the open ball centered at $x$ with radius $r>0$. The notation 
${\bf{OR}}$ represents the class of all ordinal numbers. 
In addition, $\omega$ denotes the set of all natural numbers and 
$\omega_1$ stands for the set of all countable ordinal numbers. In 
this manuscript, for all $\lambda\in{\bf OR}$, we consider its usual 
order topology given by the next topological basis 
(\cite[p. 66]{Dugundji1966})
\begin{equation*}\label{eq:TO}
	\{(\beta,\gamma):\beta,\gamma\in{\bf OR},\ \beta<\gamma\leq \lambda\}
	\cup\{[0,\beta):\beta\in{\bf OR},\ \beta\leq \lambda\}.
\end{equation*}
The following result puts in evidence the relationship between Polish 
spaces and ordinal numbers.
\begin{Proposition}\label{Prop:Polishord}
Let $\alpha$ be an ordinal number. Then, $\alpha$ is a Polish space 
if and only if $\alpha$ is countable.
\end{Proposition}
\begin{proof}
First, we suppose that the ordinal number $\alpha$ is a Polish space. We 
assume, by contradiction, that $\alpha$ is uncountable.  Thus, 
$\omega_1\leq \alpha$.  Hence,
\begin{equation*}
    \omega_1\subseteq \alpha.
\end{equation*}
Since $\alpha$ is a Polish space, we have that $\alpha$ is a separable 
space.  Then, $\omega_1$ is also separable, contradicting the fact 
that $\omega_1$ is not a separable space~(\cite[p. 114]{Willard1970}). 
Therefore, $\alpha$ is countable.

Conversely, let $\alpha$ be a countable ordinal number. By 
Theorem 2.1 and Lemma 3.6 in~\cite{Alvarez-Merino2016}, there exists 
a countable compact set $K\subseteq\mathbb{R}$ such that 
\begin{equation*}
    K\sim \omega^\alpha + 1.
\end{equation*}
Since $K\subseteq\mathbb{R}$ is compact, we see that $K$ is a 
complete metric space.  Then, $\omega^\alpha + 1$ is a completely 
metrizable space. Since $\alpha\leq \omega^\alpha$, we have that  
\begin{equation*}
    \alpha = \left[0,\alpha\right[ \subseteq \omega^\alpha 
                                   \subseteq \omega^\alpha + 1.
\end{equation*}
Thus, $\alpha$ is an open subset of a completely metrizable space. 
By Theorem 1.1 in~\cite{Michael1986}, $\alpha$ is also a completely 
metrizable space. Moreover, since $\alpha$ is countable, we obtain 
that $\alpha$ is separable. Therefore, $\alpha$ is a Polish space.
\end{proof}
In the following section, we begin by briefly reviewing the definitions of 
Cantor-Bendixson's derivative and Cantor-Bendixson's characteristic. 
Theorem~\ref{Theo:Hausdorff} and Propositions~\ref{Prop:T1} and~\ref{Prop:metric} 
imply that for all metrizable space $(E,\tau)$, function 
$\widetilde{\mathcal{CB}}_E$, given in Remark~\ref{Remark:CB}, is well-defined 
and injective. Lemma~\ref{Lemma:ccnms} is used in the proof of 
Proposition~\ref{Prop:cp}, where it is demonstrated that for all countable 
Polish space, one can find a countable ordinal number that is greater than 
or equal to the first component of the Cantor-Bendixson characteristic of 
any compact countable subset of the previously mentioned space. Moreover, 
Proposition~\ref{Prop:T1cap} and Lemma~\ref{Lemma:union} are technical results, 
for $T_1$ and $T_2$ topological spaces, respectively, that will be used in 
the subsequent proofs. Proposition~\ref{Prop:nepcm} is useful in the proof of 
Theorem~\ref{Theo:nepcms}, which asserts that for all nonempty perfect 
complete metric space, for any countable ordinal number, and for every 
nonzero natural number, there is a compact countable subset of the space under 
consideration such that its Cantor-Bendixson characteristic is equal to the above 
couple of numbers.  Theorem~\ref{Theo:nepcms} and Lemma~\ref{Lemma:usms} imply 
Theorem~\ref{Theo:previous}, where it is shown that for every uncountable Polish 
space, $(E,\tau)$, for each countable ordinal number $\alpha$, and for all $p\in\omega\smallsetminus\{0\}$, there exists $K\in\mathcal{K}_E$ such that 
$\mathcal{CB}(K)=(\alpha,p)$.  Lastly, Theorem~\ref{Theo:Polish} gives the 
principal result of this paper, where for all Polish space, we obtain  
the cardinality of the set of all the equivalence classes, up to 
homeomorphisms, of compact countable subsets of the forenamed space. 

\section{On the cardinality of the equivalence classes, up to 
homeomorphisms, of compact countable subsets of a Polish space} 
\label{Sec:Cardinality}
The next definition was first introduced by G. Cantor in~\cite{Cantor1879}.
\begin{Definition}[Cantor-Bendixson's derivative]\label{Def:D_CB}
Let $C$ be a subset of a topological space. For a given ordinal 
number $\delta \in \bf{OR}$, we define, using Transfinite Recursion, 
the \emph{$\delta$-th derivative} of $C$, written $C^{(\delta)}$, 
as follows:
\begin{itemize}
  \item $C^{(0)}=C$,
  \item $C^{(\varepsilon+1)}=(C^{(\varepsilon)})'$, for all ordinal 
        number $\varepsilon$,
  \item $\displaystyle C^{(\lambda)}=\bigcap_{\theta<\lambda} C^{(\theta)}$, 
        for all limit ordinal number $\lambda\neq 0$,
\end{itemize}
where $D'$ denotes the \emph{derived set} of $D$, i.e., the set of all 
limit points (or accumulation points) of the subset $D$.
\end{Definition}
We now give the following definition.
\begin{Definition}[Cantor-Bendixson's characteristic] \label{Def:CBCar}
Let $A$ be a subset of a topological space such that there exists 
an ordinal number $\gamma \in \bf{OR}$ in such a way that $A^{(\gamma)}$ 
is finite. We say that $(\alpha,p)\in {\bf OR} \times\omega$ is the 
\emph{Cantor-Bendixson characteristic} of $A$ if $\alpha$ is the 
smallest ordinal number such that $A^{(\alpha)}$ is finite and 
$|A^{(\alpha)}|=p$. In this case, we write $\mathcal{CB}(A)=(\alpha,p)$.
\end{Definition}
The next theorem (see Theorem 1.1 in~\cite{Alvarez-Merino2019}, where 
its proof can also be found) was first introduced by G. Cantor 
in~\cite{Cantor1883s} for $n$-dimensional Euclidean spaces.
\begin{Theorem}\label{Theo:Hausdorff}
Let $(X,\tau)$ be a Hausdorff space. For all $K \in \mathcal{K}_X$, 
there exists $\alpha \in \omega_1$ such that $K^{(\alpha)}$ 
is a finite set.  
\end{Theorem}
The following comment is Remark 1.2 in~\cite{Alvarez-Merino2019}.
\begin{Remark} \label{R2}
Last theorem implies that if $(X,\tau)$ is a Hausdorff space and 
$K \in \mathcal{K}_X$, then $\mathcal{CB}(K)$ is well-defined and, in 
addition, $\mathcal{CB}(K) \in \omega_1 \times \omega$.
\end{Remark}
The following two results and their proofs can also be found 
in~\cite{Alvarez-Merino2019} (see Propositions 3.1 and 3.2 
in~\cite{Alvarez-Merino2019}).
\begin{Proposition} \label{Prop:T1}
Let $(X,\tau)$ be a $T_1$ topological space. For all 
$K_1, K_2 \in \mathcal{K}_X$ such that $K_1 \sim K_2$, 
we have that $\mathcal{CB}(K_1) = \mathcal{CB}(K_2)$.
\end{Proposition}
\begin{Proposition} \label{Prop:metric}
Let $(E, d)$ be a metric space.  For all $K_1, K_2 \in \mathcal{K}_E$  
such that $\mathcal{CB}(K_1) = \mathcal{CB}(K_2)$, we get 
$K_1 \sim K_2$.
\end{Proposition}
\begin{Remark} \label{Remark:CB}
By Theorem~\ref{Theo:Hausdorff} and Propositions~\ref{Prop:T1} 
and~\ref{Prop:metric} above, we obtain that for all metrizable 
space $(E,\tau)$, we can define the following injective function 
\begin{equation*}
    \setlength{\arraycolsep}{2pt}
	\begin{array}{r@{}ccl}
	    \widetilde{\mathcal{CB}}_E\colon 
	    & \mathscr{K}_E & \longrightarrow & \omega_1\times\omega\\
	    & [K] & \longmapsto 
	    & \widetilde{\mathcal{CB}}_E([K])=\mathcal{CB}(K).
    \end{array}
\end{equation*}
\end{Remark}
\begin{Definition}[Scattered topological space] \label{Def:scat}
We say that a topological space $(D, \tau)$ is \emph{scattered} if 
every nonempty closed set $A$ of $D$ has at least one point which is 
isolated in $A$.
\end{Definition}
The next lemma will be used in the proof of 
Proposition~\ref{Prop:cp} below.
\begin{Lemma} \label{Lemma:ccnms}
Let $(E,d)$ be a countable complete metric space. Then, there is a 
countable ordinal number $\alpha$ such that  
\begin{equation*}
    E^{(\alpha)}=\varnothing.
\end{equation*}
\end{Lemma}
\begin{proof}
Since $E$ is a countable complete metric space, we see that $E$ 
is scattered~(\cite[p. 415]{Kuratowski1966-2}). Then, there exists 
an ordinal number $\alpha\leq |E|$ such that 
$E^{(\alpha)}=\varnothing$~(\cite[p. 139]{Kunen1984}). As $E$ is 
a countable set, $\alpha \in \omega_1$. 
\end{proof}
\begin{Proposition} \label{Prop:cp}
Let $(E,\tau)$ be a countable Polish space. Then, there exists a 
countable ordinal number $\alpha$ such that for all 
$K\in\mathcal{K}_E$, 
\begin{equation*}
    K^{(\alpha)}=\varnothing.
\end{equation*}        
\end{Proposition}
\begin{proof}
By the last lemma, there is $\alpha \in \omega_1$ such that 
\begin{equation*}
    E^{(\alpha)}=\varnothing.
\end{equation*}
Let $K\in\mathcal{K}_E$. Since $K\subseteq E$, we obtain 
\begin{equation*}
    K^{(\alpha)}\subseteq E^{(\alpha)}=\varnothing. 
\end{equation*}
Hence, $K^{(\alpha)}=\varnothing$.
\end{proof}
\begin{Remark} \label{Remark:cp}
As a consequence of last proposition, we see that for every countable 
Polish space, there is a countable ordinal number that is an upper bound 
for the first component of the Cantor-Bendixson characteristic of any 
compact countable subset of the foregoing space.
\end{Remark}
The next result is Lemma 2.2 in~\cite{Alvarez-Merino2019}, 
where its proof is also given.
\begin{Proposition}\label{Prop:T1cap}
Let $(X,\tau)$ be a $T_1$ topological space. For all $K$ and $F$ closed 
subsets of $X$ such that $K \cap F$ = $K \cap \operatorname{int}(F)$, 
where $\operatorname{int}(F)$ is the set of all interior points of $F$, 
and for every $\alpha \in \mathbf{OR}$, we have that 
\begin{equation}\label{eq:lemtech}
  (K \cap F)^{(\alpha)} = K^{(\alpha)} \cap F.
\end{equation} 
\end{Proposition}
The next lemma is a known result, which is given for convenience 
of the reader, and it will be needed in the proof of 
Proposition~\ref{Prop:nepcm} and Theorem~\ref{Theo:nepcms} below.
\begin{Lemma}\label{Lemma:union}
Let $(E, \tau)$ be a Hausdorff topological space and let $n \in \omega$. 
If $F_0, \ldots, F_n$ are closed subsets of $E$, then for all 
ordinal number $\alpha \in {\bf{OR}}$, 
\begin{equation*}
  \left( \bigcup_{k=0}^n F_k \right)^{(\alpha)} 
  = \bigcup_{k=0}^n F_k^{(\alpha)}.
\end{equation*}
\end{Lemma}
\begin{proof}
We proceed in a similar way as in the proof of Lemma 2.1  
in~\cite{Alvarez-Merino2016}.
\end{proof}
The following proposition will be used in the proof of 
Theorem~\ref{Theo:nepcms} below, and it shows some properties 
related to the compact countable subsets of any nonempty perfect (i.e., 
it coincides with its derived set) completely metrizable space.
\begin{Proposition}\label{Prop:nepcm}
Let $(P,\tau)$ be a nonempty perfect completely metrizable space. For all 
countable ordinal number $\alpha$, for every $z\in P$, and for any $r>0$, 
there exists a set $K\in\mathcal{K}_P$ such that 
\begin{equation*}
  K\subseteq B(z,r) \qquad \text{and} \qquad K^{(\alpha)}=\{z\}.
\end{equation*}        
\end{Proposition}
\begin{proof}
Let $d$ be a compatible metric on $E$ that generates $\tau$ and such that 
$(E,d)$ is a complete metric space. We will use Transfinite Induction.
\begin{itemize}
\item We first examine the case $\alpha=0$. For all $z\in P$, and for 
every $r>0$, we see that $K:=\{z\}$ satisfies the required properties.
\item Now, let $\alpha$ be a countable ordinal number such that for any 
$x\in P$, and for all $\epsilon>0$, there is a set 
$\widetilde K\in\mathcal{K}_P$ such that $\widetilde K\subseteq B(x,\epsilon)$ 
and $\widetilde K^{(\alpha)}=\{x\}$. We will show that for all $z\in P$, and 
for every $r>0$, there exists a set $K\in\mathcal{K}_P$ such that
$K\subseteq B(z,r)$ and $K^{(\alpha+1)}=\{z\}$. In order to do this, let 
$z\in P$ and $r>0$. Since $P$ is perfect, we have that $z$ is an accumulation 
point of $P$. Then, there is a sequence $(x_n)_{n\in\omega}$ in $P\smallsetminus\{z\}$ 
such that $(d(x_n,z))_{n\in\omega}$ is strictly decreasing and $d(x_n,z) \to 0$  
as $n \to +\infty$. In addition, this last sequence can be taken in 
such a way that for all $n\in\omega$, $d(x_n,z) < r$.  For all $n\in\omega$,
we take 
\begin{equation*}
   r_n:=d(x_n,z), 
\end{equation*}
and
\begin{equation*}
  \epsilon_n := \tfrac 1 2 \min\{ r_{n-1}-r_n, r_n - r_{n+1}\} >0, 
\end{equation*}
where 
\begin{equation*}
    r_{-1} := r.
\end{equation*}
Using now the induction hypothesis and the Axiom of Countable Choice, 
there is a family $\{K_n \in \mathcal{K}_P : n\in\omega\}$ such that 
for every $n\in\omega$, 
\begin{equation*}
  K_n\subseteq B(x_n,\epsilon_n) \qquad\text{and}\qquad
  K_n^{(\alpha)}=\{x_n\}.
\end{equation*}
Furthermore, for all $n\in\omega$, we have that 
\begin{equation*}
  K_n\subseteq B(x_n,\epsilon_n) \subseteq B(z,r_{n-1})  
  \qquad\text{ and }\qquad K_{n}\subseteq 
  P \smallsetminus B(z,r_{n+1}).
\end{equation*}
We now define 
\begin{equation*}
  K:=\biguplus_{n\in\omega} K_n \uplus \{z\}.
\end{equation*}
We see that $K$ satisfies the following properties.
\begin{itemize}
  \item $K\subseteq B(z,r)$, since for all $n\in\omega$, 
  $K_n\subseteq B(z,r_{n-1})\subseteq B(z,r)$.
    
  \item $K$ is countable, since it is the countable union of countable sets.
    
  \item $K$ is compact. In fact, let $\{A_i \in \tau: i\in I\}$ be an open 
  cover of $K$. There is $j\in I$ such that $z\in A_j$. Since $A_j$ is open, 
  there exists $N\in\omega$ such that for all $n \in \omega$, 
  \begin{equation*}
    n>N \implies K_n\subseteq B(z,r_{n-1}) \subseteq A_j.
  \end{equation*}
  On the other hand, since $\displaystyle D:=\biguplus_{n=0}^N K_n$ is a finite 
  union of compact sets, it follows that D is also a compact set. Then, we can 
  extract a finite open subcover, $\{A_i \in \tau: i\in J\}$, of $D$. Hence, 
  $\{A_i \in \tau : i\in J\cup\{j\}\}$ is a finite open subcover of $K$.
    
  \item $K^{(\alpha+1)}=\{z\}$. In fact, for all $n\in\omega$, we consider 
  the set
  \begin{equation*}
    F_n :=P\smallsetminus B\left(z,\tfrac {r_n + r_{n+1}}{2}\right).
  \end{equation*}
  \underline{Claim 1}: for all $n, k \in\omega$ such that $k\le n$, 
  we have that $K_k \subseteq F_n$. \\
  In fact, let $n, k \in\omega$ be such that $k\le n$. Let 
  $x \in K_k \subseteq B(x_k,\epsilon_k)$. We assume, by contradiction, 
  that $x \in B\left(z,\tfrac {r_n + r_{n+1}}{2}\right)$. Since $k\le n$, 
  we see that $x \in B\left(z,\tfrac {r_k + r_{k+1}}{2}\right)$. Thus, 
  \begin{align*}
      r_k & := d(z,x_k)
      \leq d(z,x) + d(x,x_k)\\
      &< \frac{r_{k}+r_{k+1}}{2} + \epsilon_k\\
      &\leq  \frac{r_{k}+r_{k+1}}{2} + \frac{r_k-r_{k+1}}{2}
      =r_k,
  \end{align*}
  which is an absurd.
  
  \underline{Claim 2}: for all $n, k \in\omega$ such that $k > n$, 
  we obtain that $K_k \cap F_n=\varnothing$. \\
  In fact, let $n, k \in\omega$ be such that $k> n$. We suppose, for the 
  sake of a contradiction, that there exists $x \in K_k \cap F_n$. Thus, 
  $x\in K_k\subseteq B(x_k,\epsilon_k)$ and 
  $x \in F_n := P\smallsetminus B\left(z,\tfrac {r_n + r_{n+1}}{2}\right)$.
  Hence,  
  \begin{align*}
      \frac {r_n + r_{n+1}}{2} &\le d(z,x)
      \leq d(z,x_k) + d(x_k,x)\\
      &< r_k + \epsilon_k
      \leq  r_k + \frac{r_{k-1}-r_{k}}{2}\\
      &= \frac {r_{k-1} + r_{k}}{2},
  \end{align*}
  which contradicts the fact that $(r_m)_{m \; \in \; \omega \cup \{-1\}}$ 
  is a strictly decreasing sequence. 
  
  Using Claims 1 and 2, for all $n \in \omega$, we get
  \begin{equation*}
    K \cap F_n = \biguplus_{k\in\omega} (K_k\cap F_n) 
    \uplus \left(\{z\}\cap F_n\right)  = \biguplus_{k=0}^n K_k.
  \end{equation*}
  
  \underline{Claim 3}: for all $n \in \omega$, 
  $K \cap F_n = K \cap \operatorname{int}(F_n)$. \\
  In fact, let $n \in\omega$. Since $\operatorname{int}(F_n)\subseteq F_n$, 
  we see that  $K \cap \operatorname{int}(F_n) \subseteq K \cap F_n$. To 
  show the other implication, let 
  $x\in K \cap F_n = \displaystyle \biguplus_{k=0}^n K_k$. Then, there exists 
  $k\in \{0, \ldots, n\}$ such that $x\in K_k$ and 
  \begin{equation*}
    d(x,z) \geq \frac{r_n + r_{n+1}}{2}. 
  \end{equation*}
  We have that $d(x,z) > \frac{r_n + r_{n+1}}{2}$. Indeed, in the contrary 
  case, 
  \begin{align*}
    r_k & := d(x_k,z)\leq  d(x_k,x) + d(x,z) \\
        & < \epsilon_k + \frac{r_n + r_{n+1}}{2}.
  \end{align*}
  As $k\leq n$, we get $r_n \leq r_k$. Thus, 
  \begin{align*}
    r_k & < \frac{r_{k} - r_{k+1}}{2} + \frac{r_k + r_{n+1}}{2}\\
        & = r_k - \frac{r_{k+1}}{2} + \frac{r_{n+1}}{2}.
  \end{align*}
  Therefore, 
  \begin{equation*}
    r_{k+1} < r_{n+1}.
  \end{equation*}
  Thus, $k+1>n+1$, giving a contradiction. Hence, 
  \begin{equation*}
    \epsilon := d(x,z) - \frac{r_n + r_{n+1}}{2} > 0. 
  \end{equation*}
  We assert that $B(x,\epsilon) \subseteq F_n$. In fact, let 
  $y \in B(x,\epsilon)$. We suppose, by contradiction, that 
  $y \in B\left(z,\tfrac {r_n + r_{n+1}}{2}\right)$. Then,
  \begin{align*}
   d(x, z) & \le d(x,y) + d(y,z) \\
           & < \epsilon + \frac{r_n + r_{n+1}}{2} \\
           & = d(x,z) - \frac{r_n + r_{n+1}}{2} + \frac{r_n + r_{n+1}}{2}\\
           &  = d(x,z),
  \end{align*}
  which is a contradiction. Therefore, $x\in \operatorname{int}(F_n)$.  
  Thus, Claim 3 follows.
  
  By using now Proposition~\ref{Prop:T1cap} and Lemma~\ref{Lemma:union}, 
  for all $n\in\omega$, we get
  \begin{align*}
    K^{(\alpha+1)} \cap F_n
    & = (K \cap F_n)^{(\alpha+1)} \\
    & = \left( \biguplus_{k=0}^n K_k \right)^{(\alpha+1)}\\
    & = \biguplus_{k=0}^n K_k^{(\alpha+1)}\\
    & = \biguplus_{k=0}^n \{x_k\}'\\
    & = \biguplus_{k=0}^n \varnothing\\
    & = \varnothing.
  \end{align*}
  Thus, for all $n\in\omega$, 
  \begin{equation*}
    K^{(\alpha+1)} \subseteq P\smallsetminus F_n 
    = B\left(z,\tfrac {r_n + r_{n+1}}{2}\right). 
  \end{equation*}
  Hence, 
  \begin{equation*}
    K^{(\alpha+1)} 
    \subseteq 
    \bigcap_{n\in\omega}B\left(z,\tfrac {r_n + r_{n+1}}{2}\right)
    =\{z\}.
  \end{equation*}
  On the other hand, by using again Proposition~\ref{Prop:T1cap} 
  and Lemma~\ref{Lemma:union}, for all $n\in\omega$, we have that 
  \begin{align*}
    K^{(\alpha)} \cap F_n
    & = (K \cap F_n)^{(\alpha)} \\
    & = \left(\biguplus_{k=0}^n K_k \right)^{(\alpha)}\\
    & = \biguplus_{k=0}^n K_k^{(\alpha)}\\
    & = \biguplus_{k=0}^n \{x_k\}.
  \end{align*}
  Then, for all $n\in\omega$, $x_n\in K^{(\alpha)}$. Since 
  $(x_n)_{n\in\omega}$ converges to $z$ as $n \to +\infty$, and since 
  $(x_n)_{n\in\omega}$ is a sequence in $K^{(\alpha)} \smallsetminus\{z\}$, 
  we see that $z\in K^{(\alpha+1)}$.  Therefore, 
  \begin{equation*}
    K^{(\alpha+1)}=\{z\}.
  \end{equation*}
\end{itemize}
\item Finally, let $\lambda \neq 0$ be a countable limit ordinal number 
such that for all $\beta \in {\bf{OR}}$ with $\beta<\lambda$, we have that 
for all $x\in P$, and for every $\epsilon>0$, there exists a set 
$\widetilde K\in\mathcal{K}_P$ such that $\widetilde K\subseteq B(x,\epsilon)$ 
and $\widetilde K^{(\beta)}=\{x\}$. Next, we will show that for all 
$z\in P$, and for each $r>0$, there is a set $K\in\mathcal{K}_P$ such that 
\begin{equation*}
  K\subseteq B(z,r)  \qquad\text{ and }\qquad K^{(\lambda)}=\{z\}.
\end{equation*}
Let $z\in P$ and $r>0$. Since $P$ is a perfect set, $z \in P'$.  Thus, 
there exists a sequence $(x_n)_{n\in\omega}$ in $P\smallsetminus\{z\}$ 
satisfying that $(d(x_n,z))_{n\in\omega}$ is a strictly decreasing 
sequence of real numbers converging to $0$, and such that for all 
$n\in\omega$, $d(x_n,z) < r$.  On the other hand, there is a 
strictly increasing sequence $(\beta_n)_{n\in\omega}$ in $\omega_1$ 
such that 
\begin{equation*}
  \sup\{\beta_n:n\in\omega\} = \lambda.
\end{equation*}
Thus, for all $n\in\omega$, $\beta_n<\lambda$. Proceeding now in a 
similar fashion to the previous case, we take for all $n\in\omega$, 
\begin{equation*}
  r_n: = d(x_n,z), 
\end{equation*}
and 
\begin{equation*}
  \epsilon_n := \tfrac 1 2 \min\{ r_{n-1}-r_n, r_n - r_{n+1}\} >0,
\end{equation*}
with  
\begin{equation*}
    r_{-1} := r.
\end{equation*}
Applying the hypothesis and by the Axiom of Countable Choice, there 
exists a family $\{K_n \in\mathcal{K}_P: n\in\omega\}$ such that 
for all $n\in\omega$, 
\begin{equation*}
  K_n\subseteq B(x_n,\epsilon_n) \qquad\text{and}\qquad
  K_n^{(\beta_n)}=\{x_n\}.
\end{equation*}
Furthermore, for all $n\in\omega$, 
\begin{equation*}
  K_n\subseteq B(x_n,\epsilon_n) \subseteq B(z,r_{n-1})
  \qquad\text{ and }\qquad
  K_{n}\subseteq P\smallsetminus B(z,r_{n+1}).
\end{equation*}
We now define 
\begin{equation*}
  K:=\biguplus_{n\in\omega} K_n \uplus \{z\}.
\end{equation*}
Proceeding in a similar manner to the preceding case, we have that
$K$ satisfies the following properties.
\begin{itemize}
  \item $K\subseteq B(z,r)$.
  
  \item $K$ is countable.
  
  \item $K$ is compact.
  
  \item $K^{(\lambda)}=\{z\}$.  Actually, for all $n\in\omega$, we 
  take the set
  \begin{equation*}
    D_n:= P \smallsetminus B\left(z,\tfrac {r_n + r_{n+1}}{2}\right).
  \end{equation*}
  Proceeding similarly as in the above case, for all $n\in\omega$, we get 
  \begin{equation*}
    K \cap D_n = \biguplus_{k=0}^n K_k \qquad\text{and}\qquad
    K \cap D_n = K \cap \operatorname{int}(D_n).
  \end{equation*}
  By Proposition~\ref{Prop:T1cap} and Lemma~\ref{Lemma:union}, 
  for all $n\in\omega$, we obtain that 
  \begin{align*}
    K^{(\lambda)} \cap D_n
    & = (K \cap D_n)^{(\lambda)} \\
    & = \left( \biguplus_{k=0}^n K_k \right)^{(\lambda)}\\
    & = \biguplus_{k=0}^n K_k^{(\lambda)}\\
    & = \biguplus_{k=0}^n \varnothing\\
    & = \varnothing,
  \end{align*}
  where we have used the fact that for all $k\in\omega$,   
  \begin{equation*}
    K_k^{(\lambda)}\subseteq K_k^{(\beta_k+1)} = \{x_k\}'=\varnothing.
  \end{equation*}
  We note that the last inclusion is a consequence of Remark 1.1 
  in~\cite{Alvarez-Merino2019}.  Then, for all $n\in\omega$, 
  \begin{equation*}
    K^{(\lambda)} \subseteq P \smallsetminus D_n 
    = B\left(z, \tfrac{r_n + r_{n+1}}{2}\right). 
  \end{equation*}
  Thus,
  \begin{equation*}
    K^{(\lambda)} 
    \subseteq 
    \bigcap_{n\in\omega}B\left(z, \tfrac{r_n + r_{n+1}}{2}\right)
    =\{z\}.
  \end{equation*}
  In order to show the other inclusion, let $\beta<\lambda$.  Then, 
  there is $N\in\omega$ such that $\beta<\beta_N$. Hence, for all 
  $n\in \{N, N+1, \ldots\}$, $\beta<\beta_n$.  By using again 
  Remark 1.1 in~\cite{Alvarez-Merino2019}, for all 
  $n\in \{N, N+1, \ldots\}$, we get 
  \begin{equation*}
    \{x_n\}=K_n^{(\beta_n)} \subseteq K_n^{(\beta)}
    \subseteq K^{(\beta)}.
  \end{equation*}
  Since $(x_{N+n})_{n\in\omega}$ is a sequence in $K^{(\beta)}$ that 
  converges to $z$ as $n \to +\infty$, and since $K^{(\beta)}$ is a 
  closed subset of $P$, we have that $z\in K^{(\beta)}$. So, 
  \begin{equation*}
    z\in \bigcap_{\beta<\lambda} K^{(\beta)} =: K^{(\lambda)}. 
  \end{equation*}
  Hence,
  \begin{equation*}
    K^{(\lambda)}=\{z\}.
  \end{equation*}
\end{itemize}
\end{itemize}
Therefore, the result follows for any countable ordinal number.
\end{proof}
The following result will be employed in the proof of 
Theorem~\ref{Theo:previous} below.
\begin{Theorem}\label{Theo:nepcms}
Let $(P,d)$ be a nonempty perfect complete metric space. Then, for all 
countable ordinal number $\alpha$, and for each $p\in\omega\smallsetminus\{0\}$, 
there exists $K\in\mathcal{K}_P$ such that 
\begin{equation*}
    \mathcal{CB}(K)=(\alpha,p).
\end{equation*}
\end{Theorem}
\begin{proof}
Let $\alpha$ be a countable ordinal number and let $p\in\omega\smallsetminus\{0\}$.
We take $ x \in P = P' \neq \varnothing$. Since $x$ is an accumulation point of 
$P$, there is an infinite number of elements of $P$ inside the open ball $B(x,1)$. 
Therefore, $P$ is an infinite set. Thus, $P$ contains a subset 
\begin{equation*}
    A:=\{x_k\in P : k \in \{0, \ldots, p-1 \} \},
\end{equation*}
with $p$ elements, where for all $i, j \in \{0, \ldots, p-1 \}$ such that 
$i\neq j$, $x_i\neq x_j$. We now define 
\begin{equation*}
    r:=\tfrac 1 2 \min\{d(x_i,x_j): i, j \in \{0, \ldots, p-1 \}, i\neq j\} >0.
\end{equation*}
By Proposition~\ref{Prop:nepcm}, for every $k \in \{0, \ldots, p-1 \}$, 
there is $K_k\in \mathcal{K}_P$ such that  
\begin{equation*}
    K_k\subseteq B(x_k,r)
    \qquad\text{and}\qquad
    K_k^{(\alpha)} = \{x_k\}.
\end{equation*}
We now take the set
\begin{equation*}
    K:=\biguplus_{k=0}^{p-1} K_k.
\end{equation*}
$K$ satisfies the following properties:
\begin{itemize}
\item $K$ is countable, since it is the finite union of countable sets.
\item As $K$ is the finite union of compact sets, $K$ is compact.
\item $K^{(\alpha)}=A$. In fact, by using Lemma~\ref{Lemma:union}, we 
see that
\begin{equation*}
    K^{(\alpha)} =\biguplus_{k=0}^{p-1} K_k^{(\alpha)}
                 =\biguplus_{k=0}^{p-1} \{x_k\}=:A.
\end{equation*}
Thus, $|K^{(\alpha)}|=|A|=p$.
\end{itemize}
Hence, $K\in \mathcal{K}_P$ and 
$\mathcal{CB}(K)=(\alpha,p)$.
\end{proof}
The next lemma will be used in the proof of Theorem~\ref{Theo:previous} below. 
\begin{Lemma} \label{Lemma:usms}
Let $(E,\tau)$ be an uncountable separable metrizable space. There exists a 
nonempty perfect set $P\subseteq E$.
\end{Lemma}
\begin{proof}
Let $P$ be the set of all \textit{condensation points} of $E$, i.e., the points 
such that every open neighborhood of them contains uncountably many elements  
of $E$. Since $E$ is a separable metrizable space, we see that $E$ is second 
countable. Moreover, since $|E| > \aleph_0$ and $E$ is a second countable 
topological space, it follows that $P$ is nonempty~(\cite[p. 180]{Dugundji1966}).  
Furthermore, using the fact that $E$ is a metrizable space, we have that 
$P'=P$ (\cite[p. 252]{Kuratowski1966-1}).  Hence, $P$ is a nonempty perfect set.
\end{proof}
The next theorem generalizes Corollary 2.1 in~\cite{Alvarez-Merino2016}, 
which is valid on the real line, and it will be used in the proof of 
Theorem~\ref{Theo:Polish}, below, that is the main result of this paper.
\begin{Theorem}\label{Theo:previous}
Let $(E,\tau)$ be an uncountable Polish space. Then, for all countable ordinal
number $\alpha$, and for every $p\in\omega\smallsetminus\{0\}$, there exists
$K\in\mathcal{K}_E$ such that 
\begin{equation*}
    \mathcal{CB}(K)=(\alpha,p).
\end{equation*}
\end{Theorem}
\begin{proof}
Let $\alpha$ be a countable ordinal number and let $p\in\omega$ be such that 
$p\neq 0$. By the previous lemma, there exists $P\subseteq E$ such that 
$P$ is a nonempty perfect set. Thus, by Theorem~\ref{Theo:nepcms}, there exists
a compact countable set $K\subseteq P$ such that $\mathcal{CB}(K)=(\alpha,p)$. 
Since $K$ is a compact countable subset of $P\subseteq E$, we have that 
$K\in\mathcal{K}_E$. This completes the proof.
\end{proof}
Summarizing what we have proved so far, in this section, we obtain the 
following theorem that completely determines the cardinality of all the 
equivalence classes, up to homeomorphisms, of elements of the set of all 
compact countable subsets of a Polish space, according to the cardinality 
of this space.
\begin{Theorem}\label{Theo:Polish}
Let $(E,\tau)$ be a Polish space.
\begin{enumerate}[leftmargin=*,label=(\roman*)]
\item If $E$ is finite, then 
    \begin{equation*}
        |\mathscr{K}_E| = |E|+1.
    \end{equation*}  
\item If $E$ is countably infinite, then 
    \begin{equation*}
        |\mathscr{K}_E| = |E| = \aleph_0.
    \end{equation*}
\item If $E$ is uncountable, then 
    \begin{equation*}
        |\mathscr{K}_E| = \aleph_1.
    \end{equation*}
\end{enumerate}
\end{Theorem}
\begin{proof}
Let $(E,\tau)$ be a Polish space and let $d$ be a compatible Polish metric 
on $E$ that generates $\tau$ and such that $(E,d)$ is a complete metric space. 
We consider the following cases, taking into account the cardinality of 
the set $E$.
\begin{enumerate}[leftmargin=*,label=\textit{(\roman*)}]
\item Let $E$ be a finite set such that $|E|=n \in \omega$. Every subset of 
$E$ is also a finite set and hence, it is a compact set. Then, 
$\mathcal{K}_E=\mathcal{P}(E)$. Moreover, for all $K \in \mathcal{K}_E$, 
$\mathcal{CB}(K)=(0,|K|)$. Thus, 
\begin{equation*}
   |\mathscr{K}_E|\leq n+1.
\end{equation*}
On the other hand, since for all $m \in \{0, \ldots, n\}$, there is 
$F_m \subseteq E$ such that $|F_m| = m$, it follows that  $|\mathscr{K}_E| = n+1$. 
Therefore, 
\begin{equation*}
   |\mathscr{K}_E| = |E|+1.
\end{equation*}

\item We now consider the case when $E$ is a countable infinite set, i.e., 
$|E|=\aleph_0$. By Proposition~\ref{Prop:cp}, there is a countable 
ordinal number $\alpha$ such that for all $K \in \mathcal{K}_E$, 
$K^{(\alpha)}=\varnothing$. Thus, for every  $K\in\mathcal{K}_E$, 
if $\mathcal{CB}(K)=(\beta,p)$, then $\beta<\alpha+1$. Hence, 
\begin{equation*}
   \widetilde{\mathcal{CB}}_E(\mathscr{K}_E) \subseteq (\alpha+1)\times \omega.
\end{equation*}
Consequently, 
\begin{equation*}
   |\mathscr{K}_E| \leq |(\alpha+1)\times \omega| 
   = |\alpha+1|| \omega| = |\omega|| \omega| = \aleph_0.
\end{equation*}
On the other hand, since every finite subset of $E$ is a compact set and 
since for all natural number, there is at least a subset of $E$ with 
cardinality equals the previous natural number, we have that 
$|\mathscr{K}_E|\geq \aleph_0$.  Therefore, $|\mathscr{K}_E| = \aleph_0$, i.e.,
\begin{equation*}
   |\mathscr{K}_E| = |E| = \aleph_0.
\end{equation*}

\item We suppose now that $E$ is uncountable. By Theorem~3.3 in 
\cite{Alvarez-Merino2019}, we see that $|\mathscr{K}_E|\leq \aleph_1$. On the other 
hand, by Theorem~\ref{Theo:previous}, one obtains that for all 
$\alpha \in \omega_1$ and for every $p\in\omega\smallsetminus\{0\}$, 
there exists $F\in\mathcal{K}_E$ such that $\mathcal{CB}(F)=(\alpha,p)$. Then,
\begin{equation*}
   \omega_1\times(\omega\smallsetminus\{0\}) 
   \subseteq \widetilde{\mathcal{CB}}(\mathscr{K}_E).
\end{equation*}
Thus,
\begin{equation*}
   |\mathscr{K}_E| \geq |\omega_1\times(\omega\smallsetminus\{0\})| 
   = |\omega_1| = \aleph_1.
\end{equation*}
As a consequence,
\begin{equation*}
   |\mathscr{K}_E| = \aleph_1.
\end{equation*}
\end{enumerate}
This completes the proof of the theorem.
\end{proof}


\bibliographystyle{siam}
\bibliography{Polish}   

\end{document}